\documentclass[11pt]{article}
\usepackage[T1]{fontenc}
\usepackage{lmodern,microtype}
\usepackage[a4paper,margin=1.1in]{geometry}
\usepackage{amsmath,amssymb,amsthm,mathtools}
\usepackage{enumitem}
\usepackage[numbers,sort&compress]{natbib}
\usepackage[hidelinks]{hyperref}
\usepackage{aliascnt}
\usepackage[capitalize,nameinlink]{cleveref}

\newtheorem{theorem}{Theorem}[section]
\newaliascnt{lemma}{theorem}
\newtheorem{lemma}[lemma]{Lemma}
\aliascntresetthe{lemma}
\Crefname{lemma}{Lemma}{Lemmas}
\newaliascnt{corollary}{theorem}
\newtheorem{corollary}[corollary]{Corollary}
\aliascntresetthe{corollary}
\Crefname{corollary}{Corollary}{Corollaries}
\newaliascnt{question}{theorem}
\newtheorem{question}[question]{Question}
\aliascntresetthe{question}
\Crefname{question}{Question}{Questions}
\theoremstyle{definition}
\newaliascnt{definition}{theorem}
\newtheorem{definition}[definition]{Definition}
\aliascntresetthe{definition}
\Crefname{definition}{Definition}{Definitions}
\newaliascnt{example}{theorem}
\newtheorem{example}[example]{Example}
\aliascntresetthe{example}
\Crefname{example}{Example}{Examples}
\newaliascnt{remark}{theorem}
\newtheorem{remark}[remark]{Remark}
\aliascntresetthe{remark}
\Crefname{remark}{Remark}{Remarks}
\newcommand{\C}{\mathcal C}
\newcommand{\DNN}{\mathsf{DNN}}
\newcommand{\1}{\mathbf 1}
\newcommand{\dd}{\,\mathrm d}

\DeclareMathOperator{\Ent}{Ent}
\DeclareMathOperator{\Hom}{hom}
\providecommand{\doi}[1]{\href{https://doi.org/#1}{doi:\,\nolinkurl{#1}}}
\setlist[enumerate]{label=\textup{(\roman*)},leftmargin=*}
\hypersetup{pdftitle={Tensor Amplification and Spectral Transfer for Sidorenko-Type Inequalities},pdfauthor={Yuqi Zhao}}

\title{Tensor Amplification and Spectral Transfer\\for Sidorenko-Type Inequalities}
\author{Yuqi Zhao\thanks{Email: \texttt{yuqi.zhao012@gmail.com}.}}
\date{}

\begin{document}
\maketitle

\begin{abstract}
Let $H$ be a graph with $e$ edges and $v$ non-isolated vertices. Suppose that $t(H,W)\ge p(W)^e$ holds across a class of graphons closed under tensor powers and restriction to a common vertex set, where $p(W)$ denotes the edge density. We show that whenever $H$ is not a matching, every equality case must be regular. Quantitatively, for $p=p(W)>0$, the degree-biased vertex measure $\nu$ satisfies $(2e-v)\Ent(\nu\mid\mu)\le\log(t(H,W)/p^e)$, where $\mu$ is the underlying vertex measure and the coefficient $2e-v$ is uniformly best possible over these classes. Furthermore, in the range $v\le e$, the same hypothesis is equivalent to the spectral inequality $t(H,W)\ge \rho(W)^{2e-v}p(W)^{v-e}$ for every nonzero host in the class. Both proofs proceed by restricting tensor powers to typical sets associated with two vertex measures, obtained by weighting vertices either by their degrees or by the square of a Perron eigenfunction. Specialized to bounded doubly nonnegative kernels, these results yield a spectral characterization of Sidorenko-good graphs in the range $v\le e$, and identify Sidorenko-good forcing with regular-KNRS forcing for every non-matching Sidorenko-good graph.
\end{abstract}

\section{Introduction}

Let $H$ be a finite simple graph with $e=e(H)\ge1$ edges, and let $W$ be a graphon of edge density $p(W)$. Since the constant graphon of value $p=p(W)$ has homomorphism density $t(H,p\1)=p^e$, a graph $H$ is called \emph{Sidorenko} if
\begin{equation}\label{eq:sidorenko}
  t(H,W)\ge p(W)^e \qquad\text{for every graphon }W\,.
\end{equation}
While Erd\H{o}s and Simonovits~\cite{ErdosSimonovits1984} posed a closely related supersaturation problem, Sidorenko~\cite{Sidorenko1991,Sidorenko1993} conjectured that every bipartite graph satisfies~\eqref{eq:sidorenko}. Although OpenAI~\cite{OpenAI2026Sidorenko} gave a bipartite counterexample with 35 vertices and 66 edges, the inequality is nevertheless known to hold for many broad families, including weakly norming graphs~\cite{Hatami2010}, bipartite graphs with a vertex complete to the opposite part~\cite{ConlonFoxSudakov2010}, and graphs admitting a strong tree decomposition~\cite{ConlonKimLeeLee2018}. Other productive approaches include logarithmic inequalities~\cite{LiSzegedy2026}, Cartesian products with trees~\cite{KimLeeLee2016}, local inequalities for signed graphons~\cite{Lovasz2011}, and blow-ups~\cite{ConlonLee2021}.

In this paper, we investigate what follows once the Sidorenko inequality is known to hold on a specified class of hosts. We call a graphon class $\C$ \emph{admissible} if it is closed under tensor powers and normalized principal restrictions, where a principal restriction $W[S]$ retains the values of $W$ on $S\times S$ while endowing $S$ with its conditional probability measure. Since both operations preserve positive semidefiniteness, the class of doubly nonnegative graphons is admissible. A graph $H$ is said to be \emph{$\C$-Sidorenko} if~\eqref{eq:sidorenko} holds for every $W\in\C$. Throughout, we discard isolated vertices and write $v=v(H)$ for the number of remaining vertices.

\paragraph{Equality and degrees.}
We write $\deg_W$ for the degree function of $W$, and say that $W$ is \emph{regular} if $\deg_W=p(W)$ almost everywhere. When $H$ is a matching, equality in~\eqref{eq:sidorenko} holds trivially on every host. If \(H\) is \(\mathcal C\)-Sidorenko and two of its edges share a vertex, then every equality case in \(\mathcal C\) is regular.

\begin{theorem}\label{thm:equality-regular}
Let $\C$ be admissible and let $H$ be a $\C$-Sidorenko graph with $2e>v$. If $W\in\C$ satisfies $t(H,W)=p(W)^e$, then $W$ is regular.
\end{theorem}

We call a $\C$-Sidorenko graph \emph{$\C$-forcing} if every equality case in $\C$ is constant, and \emph{$\C$-regular-forcing} if this holds for every regular equality case in $\C$. Since $2e>v$ precisely when $H$ is not a matching, \Cref{thm:equality-regular} immediately gives the following reduction.

\begin{corollary}\label{cor:relative-forcing-regular}
For an admissible class $\C$, a non-matching $\C$-Sidorenko graph is $\C$-forcing if and only if it is $\C$-regular-forcing.
\end{corollary}

The distinction between regular and constant hosts is essential even when $H$ is a tree. Indeed, if $W$ is $p$-regular, integrating out one leaf at a time yields $t(T,W)=p^{e(T)}$ for every tree $T$. For instance, the indicator of pairs lying in the same half of a balanced bipartition provides a nonconstant regular equality case for every tree. In the color-preserving setting, Szegedy~\cite[Corollary~9.2]{Szegedy2015} reduced the Sidorenko inequality to tests on coset bigraphs, while Coregliano and Razborov~\cite[Theorems~3.1 and~8.2]{CoreglianoRazborov2021} established reductions to biregular bigraphons and, in the symmetric setting, to regular graphons. By contrast, \Cref{thm:equality-regular} places a condition directly on the original host itself.

Moreover, our proof yields a quantitative estimate on near-equality throughout the non-matching range. Write $p=p(W)>0$ and sample a vertex with probability proportional to its degree, so that $\dd\nu=(\deg_W/p)\dd\mu$. Then
\begin{equation}\label{eq:entropy-intro}
  (2e-v)\Ent(\nu\mid\mu)\le \log\frac{t(H,W)}{p^e}\,.
\end{equation}
For trees in finite hosts, this recovers the degree-sequence inequality of Dellamonica et al.~\cite[Theorem~3]{DellamonicaEtAl2012}. The coefficient $2e-v$ is already sharp for the graphons $\1_{A\times A}$ (\Cref{rem:equality-optimality}).

\paragraph{Spectral transfer.}
Let $\rho(W)$ denote the operator norm of the integral operator with kernel $W$. Testing against the constant function gives $p(W)\le\rho(W)$. We call $H$ \emph{spectrally $\C$-Sidorenko} in the range $v\le e$ if it satisfies the stronger lower bound below.

\begin{theorem}[Spectral transfer]\label{thm:main}
Let $\C$ be admissible and let $H$ have $v\le e$. Then $H$ is $\C$-Sidorenko if and only if, for every nonzero $W\in\C$,
\begin{equation}\label{eq:spectral-main}
  t(H,W)\ge \rho(W)^{2e-v}p(W)^{v-e}\,.
\end{equation}
\end{theorem}

For each nonzero host, the right-hand side of~\eqref{eq:spectral-main} is $(\rho(W)/p(W))^{2e-v}$ times $p(W)^e$. The admissibility assumption on $\C$ enters only in deriving~\eqref{eq:spectral-main} from the ordinary Sidorenko inequality. Li, Lin, Liu and Zhang~\cite[Theorem~1.1]{LiLinLiuZhang2026} proved this equivalence for bipartite graphs on finite hosts. Their work was an important
motivation for the spectral part of the present paper. For a finite graph $G$ with an edge, let $M(G)=2|E(G)|$ and let $\lambda(G)$ be its adjacency spectral radius. Their inequality reads
\begin{equation}\label{eq:finite-spectral-intro}
  \Hom(H,G)\ge\lambda(G)^{2e-v}M(G)^{v-e}\,.
\end{equation}
For the step graphon $W_G$ of an $n$-vertex graph $G$, the three quantities $t(H,W_G)$, $\rho(W_G)$, and $p(W_G)$ are $\Hom(H,G)/n^v$, $\lambda(G)/n$, and $M(G)/n^2$, respectively, so~\eqref{eq:finite-spectral-intro} carries the exact normalization of~\eqref{eq:spectral-main}.

Li, Lin, Liu and Zhang obtain their spectral bound by finding regular subgraphs in tensor powers of a finite host~\cite[Theorem~1.2]{LiLinLiuZhang2026}. To extend the equivalence to admissible graphon classes, we choose a vertex set $S_k\subseteq\Omega^k$ typical for the measure given by the square of a normalized Perron eigenfunction. The $k$th root of the internal average degree $\mu^k(S_k)^{-1}\int_{S_k\times S_k}W^{\otimes k}\dd\mu^{2k}$ tends to $\rho(W)$. The principal restriction $(W^{\otimes k})[S_k]$ belongs to $\C$, so we can apply the assumed Sidorenko inequality there.

\paragraph{The two vertex measures.}
Both arguments start from multiplicativity under tensor powers. For $p(W)>0$, writing $R_H(W)=t(H,W)/p(W)^e$ gives
\begin{equation}\label{eq:intro-tensor-ratio}
  R_H(W^{\otimes k})=R_H(W)^k\,.
\end{equation}
To exploit this identity, we consider two edge distributions. First, choosing an edge with density $W/p(W)$ relative to $\mu^2$ gives each endpoint the marginal distribution $\nu$. In the tensor power, both endpoints lie in a set typical for $\nu^{\otimes k}$ with probability tending to one. The measure of this set under $\mu^{\otimes k}$ supplies the entropy term in~\eqref{eq:entropy-intro}. Second, for the spectral bound, we sample a pair with density $W(x,y)\phi(x)\phi(y)/\rho(W)$ relative to $\mu^2$, where $\phi$ is a nonnegative unit Perron eigenfunction. Each endpoint now has density $\phi^2$ relative to $\mu$, and on the corresponding typical set the product of the Perron weights varies by only a subexponential factor. This lets us compare the Rayleigh quotient of $\phi^{\otimes k}$ with that of a single indicator. We develop this common estimate in \Cref{sec:prelim}, prove the degree and spectral conclusions in \Cref{sec:degree,sec:weak-spectral}, and apply them to bounded doubly nonnegative kernels in \Cref{sec:dnn}.

\section{Restrictions and typical sets}\label{sec:prelim}

Throughout, we work on standard atomless probability spaces $(\Omega,\mu)$ and identify functions that agree almost everywhere. A \emph{graphon} is a symmetric measurable function $W:\Omega^2\to[0,1]$. We use the standard homomorphism density~\cite{LovaszSzegedy2006,Lovasz2012},
\begin{equation*}
  t(H,W)=\int_{\Omega^{V(H)}}\prod_{uv\in E(H)}W(x_u,x_v)
                    \prod_{u\in V(H)}\dd\mu(x_u)\,.
\end{equation*}
In particular, $p(W)=t(K_2,W)=\int_{\Omega^2}W\dd\mu^2$. The integral operator associated with $W$ is given by
\[
  (T_W f)(x)=\int_\Omega W(x,y)f(y)\dd\mu(y)\,.
\]
It is compact and self-adjoint on real $L^2(\mu)$, and $\deg_W=T_W\1$. Hence
\begin{equation}\label{eq:p-le-rho}
  p(W)=\langle\1,T_W\1\rangle\le\|T_W\|_{2\to2}=\rho(W)\,.
\end{equation}
Since isolated vertices integrate out of $t(H,W)$, we continue to assume that $H$ has none. Summing vertex degrees then gives
\begin{equation}\label{eq:matching-convention}
  2e-v=\sum_{u\in V(H)}(\deg_H(u)-1)\ge0\,,
\end{equation}
with equality precisely when $H$ is a matching.

\subsection{Restriction to a common vertex set}

The tensor power $W^{\otimes k}$ is the graphon on $(\Omega^k,\mu^k)$ defined by $W^{\otimes k}(x,y)=\prod_{i=1}^kW(x_i,y_i)$. Fubini's theorem and the operator identity $T_{W^{\otimes k}}=T_W^{\otimes k}$ yield
\begin{equation}\label{eq:tensor-mult}
  t(H,W^{\otimes k})=t(H,W)^k,\qquad
  p(W^{\otimes k})=p(W)^k,\qquad
  \rho(W^{\otimes k})=\rho(W)^k\,.
\end{equation}
For a measurable subset $S\subseteq\Omega$ of measure $\alpha>0$, the \emph{normalized principal restriction} $W[S]$ is $W|_{S\times S}$ endowed with the conditional probability measure $\mu|_S/\alpha$. In particular,
\begin{equation}\label{eq:p-restriction}
  p(W[S])=\alpha^{-2}\int_{S\times S}W\dd\mu^2\,.
\end{equation}

\begin{definition}\label{def:admissible}
A graphon class $\C$, considered up to measure-preserving isomorphism, is \emph{admissible} if $W\in\C$ implies $W^{\otimes k}\in\C$ for every integer $k\ge1$ and $W[S]\in\C$ for every measurable $S$ of positive measure.
\end{definition}

For a $\C$-Sidorenko graph $H$ and a host $W\in\C$, we will apply the same restriction inequality in both settings. Given $S\subseteq\Omega^k$ of measure $\alpha=\mu^k(S)>0$, write $B=\int_{S\times S}W^{\otimes k}(x,y)\dd\mu^k(x)\dd\mu^k(y)$ and $U=(W^{\otimes k})[S]$. Restricting every vertex variable of $H$ to $S$ and using the nonnegativity of the integrand together with $U\in\C$ (by admissibility), we obtain
\begin{equation}\label{eq:restriction-bound}
  t(H,W)^k\ge\alpha^v t(H,U)\ge\alpha^v p(U)^e=\alpha^{v-2e}B^e\,.
\end{equation}
In the degree argument, we choose a set that captures almost all of the edge mass. In the spectral argument, we choose a set with large internal average degree. The following lemma supplies both sets.

\subsection{Typical sets for a vertex measure}\label{subsec:basic-facts}

Let $\sigma$ be a probability measure having bounded density $h$ with respect to $\mu$. Its relative entropy is $D=\Ent(\sigma\mid\mu)=\int h\log h\dd\mu$, with the convention $0\log0=0$ and natural logarithms. Write $h_k(x)=\prod_{i=1}^k h(x_i)$ for the density of $\sigma^k$ with respect to $\mu^k$. Under $\sigma^k$, $\log h_k$ is a sum of $k$ independent random variables each of mean $D$.

\begin{lemma}\label{lem:typical}
Let $q$ be a probability measure on $\Omega^2$ with both marginals equal to $\sigma=h\mu$, where $h$ is bounded. For $w_k=k^{2/3}$, set
\begin{equation}\label{eq:typical-set}
  S_k=\{x\in\Omega^k:e^{kD-w_k}\le h_k(x)\le e^{kD+w_k}\}\,.
\end{equation}
Then $\sigma^k(S_k)\to1$, $q^{\otimes k}(S_k\times S_k)\to1$, and
\begin{equation}\label{eq:typical-size}
  e^{-kD-w_k}\sigma^k(S_k)\le\mu^k(S_k)
    \le e^{-kD+w_k}\sigma^k(S_k)\,.
\end{equation}
\end{lemma}

\begin{proof}
First note that $\{h=0\}$ is $\sigma$-null. On $\{h>0\}$, $\log h$ has finite second moment under $\sigma$, since $u(\log u)^2$ extends continuously to $u=0$ and $h$ is bounded. Chebyshev's inequality therefore gives $\sigma^k(S_k^c)\le k\operatorname{Var}_\sigma(\log h)/w_k^2\to0$. Because both marginals of $q^{\otimes k}$ equal $\sigma^k$, a union bound yields
\begin{equation*}
  1-q^{\otimes k}(S_k\times S_k)\le2\sigma^k(S_k^c)\,,
\end{equation*}
so $q^{\otimes k}(S_k\times S_k)\to1$. Finally, integrating the two bounds on $h_k$ in~\eqref{eq:typical-set} over $S_k$ with respect to $\mu^k$ and using $\int_{S_k}h_k\dd\mu^k=\sigma^k(S_k)$ establishes~\eqref{eq:typical-size}.
\end{proof}

\section{Degree bias and equality}\label{sec:degree}

Let $p=p(W)>0$. Endow $\Omega^2$ with the edge probability measure $q$ and $\Omega$ with its common marginal $\nu$, given by
\begin{equation}\label{eq:degree-marginal}
  \dd q(x,y)=\frac{W(x,y)}p\dd\mu(x)\dd\mu(y),\qquad
  \dd\nu(x)=\frac{\deg_W(x)}p\dd\mu(x)\,.
\end{equation}
The formula for the marginal follows by integrating out one variable and invoking the symmetry of $W$. Since $W\le 1$, the density of $\nu$ is bounded by $1/p$, so \Cref{lem:typical} applies with $\sigma=\nu$.

\begin{lemma}[Degree-biased tensor amplification]\label{lem:degree-amplification}
If $W$ is nonregular and $p=p(W)>0$, there exist sets $T_k\subseteq\Omega^k$ such that $\mu^k(T_k)\to0$ and
\begin{equation*}
  \int_{T_k\times T_k} W^{\otimes k}(x,y)\dd\mu^k(x)\dd\mu^k(y)=(1-o(1))p^k\,.
\end{equation*}
Moreover, for each fixed $a\in(0,1)$, the same asymptotic formula holds for a sequence of sets of measure exactly $a$.
\end{lemma}

\begin{proof}
Let $T_k$ be the typical set defined in~\eqref{eq:typical-set} for $\sigma=\nu$. Since $\deg_W/p$ has mean one and is nonconstant, the strict convexity of $u\log u$ implies $D=\Ent(\nu\mid\mu)>0$. Hence $\mu^k(T_k)\le e^{-kD+w_k}\to0$. By~\eqref{eq:degree-marginal}, the integral in the lemma equals $p^k q^{\otimes k}(T_k\times T_k)$, and the second factor tends to one by \Cref{lem:typical}. Finally, for any fixed $a\in(0,1)$, atomlessness allows us to enlarge $T_k$ to a set of measure exactly $a$ for all sufficiently large $k$; this enlargement preserves the asymptotic formula because the total edge mass is $p^k$.
\end{proof}

\subsection{The entropy bound}\label{sec:stability}

The exponential rate of decay in~\eqref{eq:typical-size} gives a lower bound on $\log(t(H,W)/p^e)$. We write $\|\nu-\mu\|_{\rm TV}=\sup_A|\nu(A)-\mu(A)|$ for the total variation distance, where the supremum runs over all measurable sets $A\subseteq\Omega$.

\begin{theorem}[Degree entropy]\label{thm:degree-entropy-stability}
Let $\C$ be admissible, let $H$ be $\C$-Sidorenko, and put $s=2e-v>0$. For $W\in\C$ with $p=p(W)>0$, let $R=t(H,W)/p^e$ and $\dd\nu=(\deg_W/p)\dd\mu$. Then
\begin{equation}\label{eq:entropy-stability}
  \Ent(\nu\mid\mu)\le\frac{\log R}{s},\qquad
  \|\nu-\mu\|_{\rm TV}\le\sqrt{\frac{\log R}{2s}}\,.
\end{equation}
In particular, for every measurable $A\subseteq\Omega$,
\begin{equation*}
  \left|\int_{A\times\Omega}W\dd\mu^2-p\mu(A)\right|
  \le p\sqrt{\frac{\log R}{2s}}\,.
\end{equation*}
\end{theorem}

\begin{proof}
Take the typical set $S_k$ for $\nu$ and $q$ from~\eqref{eq:degree-marginal}. Write $\alpha_k=\mu^k(S_k)$, $\beta_k=q^{\otimes k}(S_k\times S_k)$, and $D=\Ent(\nu\mid\mu)$. By \Cref{lem:typical}, $\beta_k\to1$ and $0<\alpha_k\le e^{-kD+w_k}$ for all sufficiently large $k$. Since the unnormalized edge mass on $S_k\times S_k$ is $p^k\beta_k$, applying~\eqref{eq:restriction-bound} and dividing by $p^{ke}$ gives
\begin{equation*}
  R^k\ge\alpha_k^{-s}\beta_k^e\ge\exp\bigl(s(kD-w_k)\bigr)\beta_k^e\,.
\end{equation*}
Taking logarithms, dividing by $k$, and letting $k\to\infty$ yields $\log R\ge sD$. Pinsker's inequality $\|\nu-\mu\|_{\rm TV}^2\le D/2$ then gives the second estimate in~\eqref{eq:entropy-stability}, and the final assertion follows from the identity $\int_{A\times\Omega}W\dd\mu^2=p\nu(A)$.
\end{proof}

\begin{proof}[Proof of \Cref{thm:equality-regular,cor:relative-forcing-regular}]
If $p(W)=0$, then $W=0$ almost everywhere, which is trivially regular. If $p(W)>0$ and equality holds, then $R=1$. By~\eqref{eq:entropy-stability}, $\Ent(\nu\mid\mu)=0$, so $\nu=\mu$ and hence $\deg_W=p(W)$ almost everywhere. Thus every equality case is regular. Consequently, the implication from equality to constancy can be tested solely on regular graphons in $\C$, establishing \Cref{cor:relative-forcing-regular}.
\end{proof}

\begin{remark}[Sharpness]\label{rem:equality-optimality}
Fix a measurable set $A$ with $\mu(A)=a\in(0,1)$ and take $W=\1_{A\times A}$. Since every vertex of $H$ lies on at least one edge, $p=a^2$ and $t(H,W)=a^v$. The degree-biased measure is $\mu(\,\cdot\mid A)$, so $\Ent(\nu\mid\mu)=\log(1/a)$ and $\log R=(2e-v)\log(1/a)$. Thus the coefficient $2e-v$ in~\eqref{eq:entropy-intro} cannot be increased, even when restricted to doubly nonnegative hosts. On the other hand, for a matching $mK_2$, one has $t(mK_2,W)=p(W)^m$ for every $W$, which explains the necessity of the strict condition $2e>v$ in \Cref{thm:equality-regular}.
\end{remark}

\paragraph{The finite-host entropy inequality.}
When $H$ is Sidorenko on all graphons, the finite-host case of~\eqref{eq:entropy-intro} can also be deduced from Szegedy's entropy characterization~\cite[Corollary~9.1]{Szegedy2015}. First suppose $H$ is Sidorenko in the color-preserving sense, meaning that the inequality holds for all bigraphons. A uniformly oriented edge of an $n$-vertex host $G$ with $m>0$ edges has entropy $\log(2m)$, and each endpoint follows the degree-biased law $\nu$ of Shannon entropy $h(\nu)$. Szegedy's inequality yields a distribution on homomorphisms from $H$ to $G$ of entropy at least $e\log(2m)-(2e-v)h(\nu)$, which is bounded above by $\log\Hom(H,G)$. Substituting $p=2m/n^2$ and $\Ent(\nu\mid\mu)=\log n-h(\nu)$, where $\mu$ is uniform on $V(G)$, recovers~\eqref{eq:entropy-intro}.

For the ordinary Sidorenko property, $H$ must be bipartite, as seen by testing on a complete bipartite host. Fix a bipartition, denote the resulting bigraph by $B$, and write $B^{\rm op}$ for its reversal. Then the disjoint union $B\sqcup B^{\rm op}$ is Sidorenko in the color-preserving sense: indeed, for any bigraphon $U:X\times Y\to[0,1]$, the symmetric graphon $V$ on $X\times Y$ defined by $V((x,y),(x',y'))=U(x,y')U(x',y)$ satisfies
\begin{equation*}
  t(B\sqcup B^{\rm op},U)=t(H,V)\ge p(V)^e=p(U)^{2e}\,,
\end{equation*}
where $p(U)$ denotes the integral of $U$ against the product measure. On a symmetric host, the two bigraphs $B$ and $B^{\rm op}$ have the same homomorphism count as $H$; applying the preceding entropy estimate to their disjoint union and dividing by two gives the claim. For a general admissible class, however, the direct proof of \Cref{thm:degree-entropy-stability} has the advantage of using only hosts within that class.

\section{Spectral regularization and transfer}\label{sec:weak-spectral}

For an indicator $\1_S$, the Rayleigh quotient of the graphon operator equals the unnormalized edge mass on $S\times S$ divided by the measure of $S$. While a single indicator on $\Omega$ can have a strictly smaller Rayleigh quotient than the Perron eigenfunction, we show that in the $k$th tensor power one can choose an indicator whose quotient comes within a subexponential factor of the operator norm.

\subsection{Principal sets in tensor powers}\label{subsec:operator-facts}

\begin{lemma}\label{lem:perron-eigenfunction}
Every nonzero graphon $W$ has a bounded nonnegative eigenfunction $\phi$ such that $\|\phi\|_2=1$ and $T_W\phi=\rho(W)\phi$.
\end{lemma}

\begin{proof}
Since $T_W$ is compact and self-adjoint, it admits a real unit eigenfunction $\psi$ corresponding to an eigenvalue of modulus $\rho=\rho(W)>0$. Because $W\ge0$, we have
\begin{equation*}
  \rho=|\langle\psi,T_W\psi\rangle|
  \le\langle|\psi|,T_W|\psi|\rangle\le\rho\,.
\end{equation*}
Setting $\phi=|\psi|$ in the equality case of Cauchy--Schwarz gives $T_W\phi=\rho\phi$. The eigenvalue equation then bounds $\phi$ almost everywhere by $\rho^{-1}\|W\|_\infty\|\phi\|_1$.
\end{proof}

\begin{theorem}[Weak spectral regularization]\label{thm:weak-spectral}
Let $W$ be a nonzero graphon. For $k\ge1$, set
\begin{equation*}
  D_k(W)=\sup_{\mu^k(S)>0}\frac{1}{\mu^k(S)}
      \int_{S\times S} W^{\otimes k}(x,y)\dd\mu^k(x)\dd\mu^k(y)\,,
\end{equation*}
where $S$ ranges over measurable subsets of $\Omega^k$. Then $\displaystyle\lim_{k\to\infty}D_k(W)^{1/k}=\rho(W)$.
\end{theorem}

\begin{proof}
Write $\rho=\rho(W)$. For every measurable $S$ of positive measure, its Rayleigh quotient is bounded above by the operator norm of $T_{W^{\otimes k}}$, so
\begin{equation}\label{eq:Dk-upper}
  D_k(W)\le\rho(W^{\otimes k})=\rho^k\,.
\end{equation}
For the lower bound, choose $\phi$ as in \Cref{lem:perron-eigenfunction} and define
\begin{equation}\label{eq:pi-q-graphon}
  \dd\pi(x)=\phi(x)^2\dd\mu(x),\qquad
  \dd q(x,y)=\frac{W(x,y)\phi(x)\phi(y)}{\rho}\dd\mu(x)\dd\mu(y)\,.
\end{equation}
The eigenvalue equation $T_W\phi=\rho\phi$ shows that the first marginal of $q$ is $\rho^{-1}\phi\,T_W\phi\dd\mu=\dd\pi$; symmetry gives the second marginal, and $\|\phi\|_2=1$ ensures that both $\pi$ and $q$ are probability measures. This is the Perron edge distribution used for finite graphs in~\cite[Section~2.2]{LiLinLiuZhang2026}.

Now apply \Cref{lem:typical} with $h=\phi^2$. Write $D=\Ent(\pi\mid\mu)$ and $\Phi_k(x)=\prod_{i=1}^k\phi(x_i)$, so that $h_k=\Phi_k^2$. On $S_k\times S_k$, we have $\Phi_k(x)\Phi_k(y)\le e^{kD+w_k}$. Taking the $k$-fold product of the edge measure in~\eqref{eq:pi-q-graphon} gives
\begin{equation*}
  q^{\otimes k}(S_k\times S_k)
    =\rho^{-k}\int_{S_k\times S_k}W^{\otimes k}(x,y)\Phi_k(x)\Phi_k(y)
                                  \dd\mu^k(x)\dd\mu^k(y)\,.
\end{equation*}
Consequently, the edge mass $B_k$ of $W^{\otimes k}$ on $S_k\times S_k$ satisfies $B_k\ge \rho^k e^{-kD-w_k}q^{\otimes k}(S_k\times S_k)$. By~\eqref{eq:typical-size}, $\alpha_k=\mu^k(S_k)$ is positive for all sufficiently large $k$ and satisfies $\alpha_k\le e^{-kD+w_k}$. The factors $e^{-kD}$ cancel upon division, yielding
\begin{equation}\label{eq:dbar-pre-absorb}
  D_k(W)\ge\frac{B_k}{\alpha_k}
           \ge \rho^k e^{-2w_k}q^{\otimes k}(S_k\times S_k)\,.
\end{equation}
Since $q^{\otimes k}(S_k\times S_k)\to1$ and $w_k=o(k)$, taking $k$th roots and combining with~\eqref{eq:Dk-upper} completes the proof.
\end{proof}

\subsection{The spectral inequality}\label{sec:spectral-transfer}

For a set $S\subseteq\Omega^k$ of measure $\alpha>0$, let $B$ be its internal edge mass as in~\eqref{eq:restriction-bound}. By \Cref{thm:weak-spectral}, the supremum of $B/\alpha$ over $S$ has $k$th root tending to $\rho(W)$. We also have $B\le p(W)^k$. The sign of $v-e$ determines when these two estimates give a lower bound on $t(H,W)$.

\begin{proof}[Proof of \Cref{thm:main}]
Suppose first that $H$ is $\C$-Sidorenko, and fix a nonzero $W\in\C$ with $p=p(W)>0$. For any measurable $S\subseteq\Omega^k$ of measure $\alpha>0$ and positive edge mass $B$, we have $B\le p^k$. Since $v-e\le0$, rewriting~\eqref{eq:restriction-bound} gives
\begin{equation*}
  t(H,W)^k\ge \alpha^{v-2e}B^e
     =\left(\frac B\alpha\right)^{2e-v} B^{v-e}
     \ge\left(\frac B\alpha\right)^{2e-v}p^{k(v-e)}\,.
\end{equation*}
Because the whole space $\Omega^k$ has positive edge mass, sets of zero edge mass do not affect the supremum defining $D_k(W)$. Taking the supremum over $S$ and then taking $k$th roots yields
\begin{equation*}
  t(H,W)\ge \bigl(D_k(W)^{1/k}\bigr)^{2e-v}p^{v-e}\,.
\end{equation*}
Letting $k\to\infty$ and invoking \Cref{thm:weak-spectral} establishes~\eqref{eq:spectral-main}. Conversely, since $\rho(W)\ge p$, the right-hand side of~\eqref{eq:spectral-main} is at least $p^e$ for every nonzero $W\in\C$, while the ordinary inequality holds trivially for the zero graphon because $e\ge1$.
\end{proof}

\begin{example}\label{ex:path}
The path $P_3$ on three vertices is Sidorenko, as $t(P_3,W)=\int\deg_W^2\dd\mu\ge p(W)^2$. Li, Lin, Liu and Zhang~\cite[after Theorem~1.1]{LiLinLiuZhang2026} observed that its spectral inequality fails. The following example shows that it can fail even on a doubly nonnegative graphon. Partition $\Omega=A\sqcup B$ with $a=\mu(A)>b=\mu(B)>0$, and take $W=\1_{A\times A}+\1_{B\times B}$. Its two nonzero eigenvalues are $a$ and $b$, while $p(W)=a^2+b^2$. Thus
\begin{equation*}
  t(P_3,W)=a^3+b^3<a(a^2+b^2)=\rho(W)p(W)\,.
\end{equation*}
Here $v=e+1$, showing that spectral transfer can fail immediately outside the range of \Cref{thm:main}.
\end{example}

\begin{corollary}[$L^2$-stability of degrees]\label{cor:perron-l2-stability}
Let $\C$ be admissible, let $H$ be $\C$-Sidorenko with $v\le e$, and put $s=2e-v$. For every nonzero $W\in\C$, with $p=p(W)$ and $R=t(H,W)/p^e$,
\begin{equation*}
  \left\|\frac{\deg_W}{p}-\1\right\|_2^2\le R^{2/s}-1\,.
\end{equation*}
In particular, every measurable $A\subseteq\Omega$ satisfies
\begin{equation*}
  \left|\int_{A\times\Omega}W\dd\mu^2-p\mu(A)\right|
   \le p\sqrt{\mu(A)(1-\mu(A))}\sqrt{R^{2/s}-1}\,.
\end{equation*}
\end{corollary}

\begin{proof}
By \Cref{thm:main}, we have $\rho(W)/p\le R^{1/s}$. Since $\int\deg_W\dd\mu=p$ and $\|\deg_W\|_2=\|T_W\1\|_2\le\rho(W)$, expanding the $L^2$ norm gives
\begin{equation*}
  \left\|\frac{\deg_W}{p}-\1\right\|_2^2
    =\frac{\|\deg_W\|_2^2}{p^2}-1\le R^{2/s}-1\,.
\end{equation*}
Because $f=\deg_W/p-\1$ has mean zero, we can write $\int_A f\dd\mu=\langle\1_A-\mu(A)\1,f\rangle$, and Cauchy--Schwarz yields the second inequality.
\end{proof}

\section{Doubly nonnegative kernels}\label{sec:dnn}

A bounded symmetric kernel $K$ is \emph{positive semidefinite}, written $K\succeq0$, if $\langle f,T_K f\rangle\ge0$ for every real $f\in L^2(\mu)$. A \emph{doubly nonnegative} kernel is one that is both nonnegative almost everywhere and positive semidefinite. Following the spectral formulation of Sidorenko~\cite{Sidorenko2021DNN} and the terminology of~\cite{Zhao2026LpKNRS}, we write $\DNN$ for the class of doubly nonnegative graphons. For any bounded nonnegative kernel $K$, the same integrals define $t(F,K)$ and $p(K)=\|K\|_1$.

\begin{lemma}\label{lem:dnn-admissible}
The class $\DNN$ is admissible.
\end{lemma}

\begin{proof}
If $T_W\succeq0$, let $A=T_W^{1/2}$; then $T_{W^{\otimes k}}=(A^{\otimes k})^* A^{\otimes k}$ is positive semidefinite. Given a measurable set $S$ of measure $\alpha>0$ and $f\in L^2(S,\mu|_S/\alpha)$, extend $f$ by zero to $\widetilde f$ on $\Omega$. The quadratic form of the normalized restriction satisfies
\begin{equation*}
  \langle f,T_{W[S]}f\rangle_{L^2(S,\mu|_S/\alpha)}
     =\alpha^{-2}\langle \widetilde f,T_W\widetilde f\rangle_{L^2(\mu)}\ge0\,.
\end{equation*}
Since both operations also preserve the range $[0,1]$, $\DNN$ is admissible.
\end{proof}

\begin{definition}\label{def:sid-good}
A graph $F$ is \emph{Sidorenko-good} if $t(F,K)\ge\|K\|_1^{e(F)}$ for every bounded doubly nonnegative kernel $K$. Such a graph is \emph{Sidorenko-good-forcing} if equality for a nonzero bounded doubly nonnegative kernel implies that $K=\|K\|_1\1$ almost everywhere.
\end{definition}

\begin{corollary}\label{cor:sid-good-equivalence}
Let $F$ have $v(F)\le e(F)$. Then $F$ is Sidorenko-good if and only if, for every nonzero bounded doubly nonnegative kernel $K$,
\begin{equation*}
  t(F,K)\ge\rho(K)^{2e(F)-v(F)}\|K\|_1^{v(F)-e(F)}\,.
\end{equation*}
\end{corollary}

\begin{proof}
For $K\ne0$, the normalized kernel $W=K/\|K\|_\infty$ belongs to $\DNN$. Since both the ordinary and spectral inequalities are homogeneous of degree $e(F)$ in $K$, the equivalence follows directly from \Cref{lem:dnn-admissible,thm:main}, while the ordinary inequality is immediate at $K=0$.
\end{proof}

\subsection{Regular local density and forcing}

A graphon $W$ is \emph{$r$-locally dense}, for $r\in[0,1]$, if
\begin{equation}\label{eq:locally-dense}
  \int_{S\times S}W\dd\mu^2\ge r\mu(S)^2
  \qquad\text{for every measurable } S\subseteq\Omega\,.
\end{equation}
We call $W$ \emph{$r$-regular} if $\deg_W=r$ almost everywhere. On regular graphons, positive semidefiniteness and local density at the edge density specify the same class of hosts.

\begin{definition}\label{def:regular-knrs}
A graph $F$ is \emph{regular-KNRS} if $t(F,W)\ge r^{e(F)}$ for every $r\in[0,1]$ and every $r$-regular, $r$-locally dense graphon $W$. It is \emph{regular-KNRS-forcing} if it is regular-KNRS and, for $0<r<1$, equality implies $W=r\1$ almost everywhere.
\end{definition}

\begin{lemma}\label{lem:regld-dnn}
\begin{enumerate}[label=\textup{(\alph*)}]
\item If $W$ is $r$-regular and $r$-locally dense, then $W-r\1\succeq0$. In particular, $W\in\DNN$.
\item If $K$ is a bounded doubly nonnegative kernel of mean $p$ and $\deg_K=p$ almost everywhere, then $K-p\1\succeq0$. Consequently, $\int_{S\times S}K\dd\mu^2\ge p\mu(S)^2$ for every measurable $S$.
\end{enumerate}
\end{lemma}

\begin{proof}
For (a), the centered-kernel conclusion $W-r\1\succeq0$ is \cite[Corollary~2.15]{BradacSudakovWigderson2024}, and adding the positive semidefinite constant kernel $r\1$ yields $W\succeq0$. For (b), decompose any $f\in L^2(\mu)$ as $f=c\1+g$ with $g\perp\1$. Since $T_K\1=p\1$, the operator $T_{K-p\1}$ annihilates constants and agrees with $T_K$ on $\1^\perp$. Thus $\langle f,T_{K-p\1}f\rangle=\langle g,T_Kg\rangle\ge0$, and testing this centered kernel on $f=\1_S$ yields the final inequality.
\end{proof}

By part~(a), every Sidorenko-good graph is regular-KNRS, as noted in~\cite[Proposition~1.4(3)]{Zhao2026LpKNRS}. For a non-matching Sidorenko-good graph, degree regularity at equality provides the converse implication for forcing.

\begin{theorem}\label{thm:sidgood-regknrs-forcing}
Let $F$ be Sidorenko-good and assume it is not a matching.
\begin{enumerate}
\item \label{item:equality-regld} If a nonzero bounded doubly nonnegative kernel $K$ with $p=\|K\|_1$ satisfies $t(F,K)=p^{e(F)}$, then $K$ is $p$-regular and $K-p\1\succeq0$.
\item \label{item:regular-knrs} The graph $F$ is regular-KNRS.
\item \label{item:forcing-equiv} The graph $F$ is Sidorenko-good-forcing if and only if it is regular-KNRS-forcing.
\end{enumerate}
\end{theorem}

\begin{proof}
For~\ref{item:equality-regld}, scaling $K$ by a positive constant produces a graphon in $\DNN$ while preserving equality by homogeneity. By \Cref{thm:equality-regular,lem:dnn-admissible}, this scaled graphon is regular, so $K$ is $p$-regular; part~(b) of \Cref{lem:regld-dnn} then yields $K-p\1\succeq0$.

Next, if $W$ is $r$-regular and $r$-locally dense, then \Cref{lem:regld-dnn}(a) gives $W\in\DNN$ and $\|W\|_1=r$. The Sidorenko-good inequality therefore proves~\ref{item:regular-knrs}, and on the same hosts, Sidorenko-good forcing gives the forward implication of~\ref{item:forcing-equiv}.

Conversely, suppose $F$ is regular-KNRS-forcing and $K$ is a nonzero bounded doubly nonnegative equality case. Put $p=\|K\|_1$, choose $M>\|K\|_\infty$, and set $W=K/M$ and $r=p/M\in(0,1)$. By~\ref{item:equality-regld}, $W$ is $r$-regular and $W-r\1\succeq0$; testing on indicators shows that $W$ is also $r$-locally dense. Since $t(F,W)=r^{e(F)}$, regular-KNRS forcing implies $W=r\1$, and hence $K=p\1$ almost everywhere.
\end{proof}

\section{Concluding remarks}\label{sec:concluding}

Does every Sidorenko graph containing a cycle force quasirandomness? This question is also discussed by Nie and Wei~\cite{NieWei2026}. By \Cref{cor:relative-forcing-regular}, it suffices to prove that for every such $H$ and every nonconstant regular graphon $W$ of density $p$, one has $t(H,W)>p^{e(H)}$. On such a host, the degree-biased vertex measure coincides with the original measure $\mu$, so the entropy on the left-hand side of~\eqref{eq:entropy-intro} vanishes. Any further application of tensor powers would therefore need to capture structural features of $W$ beyond its degree distribution.

A related question concerns the choice of host construction. For tensor powers, the identity~\eqref{eq:intro-tensor-ratio} yields $k^{-1}\log R_H(W^{\otimes k})=\log R_H(W)$ whenever the ratio is positive. For a bipartite pattern $H$ with fixed vertex classes, Szegedy~\cite[Theorem~4]{Szegedy2015} showed that the supremum of $\log t(H,G)/\log p(G)$ over bigraphs $G$ with $0<p(G)<1$ can be recovered by testing on coset bigraphs. Here $t(H,G)$ counts homomorphisms that preserve the vertex classes, with the usual density normalization. This naturally prompts asking which other operations preserve analogous comparisons while revealing deeper structure in the host.

\begin{question}
For which algebraic constructions $W\mapsto U_k(W)$, with a positive scale $n_k$ specified by the construction, does one have
\begin{equation*}
  \limsup_{k\to\infty}\frac{1}{n_k}\log R_H(U_k(W))
    \le \log R_H(W)
\end{equation*}
for every $W$ in a given graphon class for which the ratios are defined and positive? For which families of graphs $H$ can the same construction and scale be used?
\end{question}

\paragraph{Acknowledgments.}
I thank Hong Liu and Seonghyuk Im for their suggestions in early 2025 on the presentation of my earlier proof that equality in Sidorenko's inequality forces regularity for non-matching Sidorenko graphs. That proof used essentially the same degree-bias argument as in \Cref{lem:degree-amplification} to concentrate almost all the edge mass of a tensor power on a set of measure tending to $1/2$.

\paragraph{AI assistance.} ChatGPT was used to assist with parts of the proofs and with the exposition. Gemini was used to assist with the exposition.

\bibliographystyle{plainnat}
{\small
\setlength{\bibsep}{0.15\baselineskip}
\bibliography{ref}
}
\end{document}